\documentclass[11pt]{amsart}

\usepackage{amssymb,latexsym,amsmath,amsfonts}
\usepackage[mathscr]{eucal}
\usepackage{mathrsfs}
\usepackage{graphicx}
\usepackage{amscd}
\usepackage{amsthm}
\usepackage{amsxtra}
\usepackage[nointegrals]{wasysym}
\usepackage{bm}
\usepackage{mathtools}
\usepackage{hyperref}
\usepackage[usenames,dvipsnames]{xcolor}
\usepackage[normalem]{ulem}
\usepackage{lineno}
\pdfoutput=1

\numberwithin{equation}{section}
\theoremstyle{definition}
\newtheorem{definition}{Definition}[section]
\theoremstyle{definition}

\theoremstyle{plain}
\newtheorem{theorem}[definition]{Theorem}
\newtheorem{lemma}[definition]{Lemma}

\newtheorem{Prop}[definition]{Proposition}

\newcommand{\beas}{\begin{eqnarray*}}
\newcommand{\eeas}{\end{eqnarray*}}
\newcommand{\bes} {\begin{equation*}}
\newcommand{\ees} {\end{equation*}}
\newcommand{\be} {\begin{equation}}
\newcommand{\ee} {\end{equation}}
\newcommand{\bea} {\begin{eqnarray}}
\newcommand{\eea} {\end{eqnarray}}

\newcommand{\eps}{\varepsilon}

\newcommand{\de} {\delta}
\newcommand{\om}{\omega}

\newcommand{\Om}{\Omega}

\newcommand{\cont}{\mathcal{C}}
\newcommand{\hol}{\mathcal{O}}

\newcommand{\wt}{\widetilde}

\newcommand{\CC}{\mathbb{C}^2}
\newcommand{\Cn}{\mathbb{C}^n}

\newcommand{\rl}{\mathbb{R}}
\newcommand{\Z} {\mathbb{Z}}
\newcommand{\N} {\mathbb{N}}

\title{Hypersurface convexity and extension of K\"ahler forms}
\author{Blake J. Boudreaux}
\address[Blake J. Boudreaux]{The University of Western Ontario}
\email{bboudre7@uwo.ca}

\author{Purvi Gupta}
\address[Purvi Gupta]{Indian Institute of Science}
\email{purvigupta@iisc.ac.in}

\author{Rasul Shafikov}
\address[Rasul Shafikov]{The University of Western Ontario}
\email{shafikov@uwo.ca}

\begin{document}
\begin{abstract}
The following generalization of a result of  S. Nemirovski is proved:  if $X$ is either a projective or a Stein manifold and $K\subset X$ is a compact sublevel set of a strictly plurisubharmonic function $\varphi$ defined in a neighborhood of $K$, then $X\setminus K$ is a union of positive divisors
if and only if  $dd^c\varphi$ extends to a Hodge form on $X$. For an arbitrary compact subset $K\subsetneq X$, this gives that $X\setminus K$ is a union of positive divisors if and only if $K$
admits a neighbourhood basis of sublevel sets of strictly plurisubharmonic functions with the $dd^c$-extension property.
\end{abstract}

\maketitle

\section{Introduction}
A compact set $K$ in $\mathbb C^n$ is called rationally convex if for any point $z \in \mathbb C^n \setminus K$, there exists a complex (algebraic) hypersurface that passes through $z$ and avoids $K$. It is generally a difficult problem to determine whether a given compact set is rationally convex. A striking characterization for totally real submanifolds was obtained by Duval--Sibony~\cite{DuSi95}: a compact totally real submanifold $M\subset \mathbb C^n$ is rationally convex if and only if there exists a K\"ahler form {\color{black} $\omega$ on $\mathbb C^n$ with respect to which $M$ is isotropic (Lagrangian if $\dim M = n$), i.e., $\iota_M^*\omega = 0$, where $\iota_M : M \to \mathbb C^n$ is the inclusion map}. A variation of this was proved by Nemirovski~\cite{Ne08}: if a compact set $K \subset \mathbb C^n$ is given as a sublevel set of a strictly plurisubharmonic function $\varphi$ defined on a neighborhood of $K$, then $K$ is rationally convex if and only if there exists a K\"ahler form $\omega$ on $\mathbb C^n$ such that $\omega$ agrees with $dd^c \varphi$ on some neighborhood of $K$. In particular, this characterization can be applied to closures of bounded strongly pseudoconvex domains in $\mathbb C^n$.

Rational convexity can be generalized to general complex manifolds in different ways. Following Guedj \cite{Gu99}, we say that a compact set $K$ on a projective manifold $X$ is rationally convex if
for any point $p\in X\setminus K$, there exists a positive divisor {\color{black} (see Section~\ref{S:prelim1})} passing though $p$ and avoiding $K$. When $X$ is a Stein manifold, we consider two competing notions of convexity. We say that a compact set $K \subset X$  is convex with respect to hypersurfaces if for any $p\in X \setminus K$, there exists a complex hypersurface $\Sigma$ that passes through $p$ but avoids $K$.
We say that $K$ is  is convex with respect to principal hypersurfaces if $\Sigma$ can be given as the zero set of an entire function on $X$. See Section~\ref{S:prelim1} for further discussion of various notions of convexity on complex manifolds. The goal of this paper is to generalize Nemirovski's characterization of convexity to projective and to Stein manifolds.

\begin{theorem}\label{T:main}
	Let $X$ be a projective manifold and $\varphi$ be a smooth strictly plurisubharmonic function on an open set $U\subseteq X$. The compact set $K=\{z\in U\,:\,\varphi(z)\leq 0\}$ is rationally convex if and only if $dd^c\varphi$ extends off a neighborhood of $K$ to a Hodge form on $X$.
\end{theorem}

\begin{theorem}\label{T:stein}
Let $X$ be a Stein manifold and $\varphi$ be a smooth strictly plurisubharmonic function on an open set $U\subseteq X$. The compact set $K=\{z\in U\,:\,\varphi(z)\leq 0\}$ is convex with respect to (principal) hypersurfaces if and only if $dd^c\varphi$ extends off a neighborhood of $K$ to a (trivial) Hodge form on $X$.
\end{theorem}

The assumption in the above theorems that $K$ is a sublevel set of a strictly plurisubharmonic function may seem quite strong. However, it can be used to give a characterization of {\it all}
compact sets that are convex with respect to (principal) hypersurfaces, as our next result shows.

\begin{theorem}\label{T:char}
Let $X$ be a Stein manifold, and $K\subset X$ be a compact set.
Then $K$ is convex with respect to (principal) hypersurfaces if and only if there exists a neighborhood basis of $K$ such that every element of the basis is of the form  $\Omega=\{\rho < 0\}$, where $\rho$ is a strictly plurisubharmonic function in a neighborhood of $\overline \Omega$, and $dd^c \rho$ extends off a neighborhood of $\overline\Om$ to a (trivial) Hodge form on $X$. The same characterization holds for rational convexity of proper compact subsets of a projective manifold.
\end{theorem}

The connection between rational convexity and K\"ahler forms was already established by Duval--Sibony~\cite{DuSi95}. The above characterization in $\Cn$ was mentioned as a remark by Nemirovski~\cite{Ne08}. The crux of the matter is that a rationally convex set admits a fundamental system of Stein neighborhoods, each of which has rationally convex closure in the ambient space. This can then be combined with Theorems~\ref{T:main} and~\ref{T:stein} to, in fact, produce a fundamental system of strongly pseudoconvex neighborhoods with rationally convex closures.


Historically, the interest in rational convexity stems from an approximation result known as the Oka--Weil theorem. We have the following version of this theorem for projective manifolds. {\color{black} In the case of $\mathbb{CP}^n$ we recover a result of Hirschowitz~\cite[Theorem 5]{Hi71}}. For Oka--Weil-type theorems in Stein manifolds, see \cite{Ro61}, \cite{Hi71} and \cite{BS}.

\begin{theorem}\label{T:OkaWeil}
Let $X$ be a complex projective manifold and $K \subset X$ be a rationally convex compact set. Then every holomorphic function in a neighborhood of $K$ is the uniform limit on $K$ of a sequence of meromorphic functions on $X$ of the form $f/g$, where $f,g$ are global holomorphic sections of a positve line bundle on $X$, with poles off $K$.
\end{theorem}

Recently rational convexity has attracted attention in symplectic geometry and topology.  It was shown by Eliashberg~\cite{El90} and Eliashberg--Cieliebak~\cite{ElCe} that, for the closure $W$ of a smoothly bounded domain in $\mathbb C^n$, $n\ge 3$,  the following conditions are equivalent.
\begin{itemize}
\item [$(a)$] $W$ admits a defining Morse function having no critical points of index greater than $n$.
\item [$(b)$] $W$ is smoothly isotopic to the closure of a strongly pseudoconvex domain in $\Cn$.
\item [$(c)$] $W$ is smoothly isotopic to a ``rationally convex domain", i.e., a rationally convex set which is the closure of a strongly pseudoconvex domain in $\Cn$.
\item [$(d)$] $W$ is isotopic to a Weinstein domain symplectically embedded in $(\mathbb R^{2n},\omega_{std})$.
\end{itemize}
When $n=2$, the question is more delicate as demonstrated by  Nemirovski and Siegel~\cite{NeSi16} who produce examples of disk bundles over surfaces that admit embeddings in $\CC$ as closures of strongly pseudoconvex domains, but not as rationally convex domains. In \cite{Go13}, Gompf considered the case of domains with trivial topology, and produced examples of homology spheres that embed in $\CC$ as boundaries of contractible strongly pseudoconvex domains. He conjectured that no Brieskorn homology sphere is realizable as the boundary of a strongly pseudoconvex domain in $\CC$. While Gompf's conjecture is still open, Mark and Tosun~\cite{MaTo} show that no Brieskorn homology sphere is orientation-preserving diffeomorphic to the boundary of a rationally convex domain in $\CC$. Their proof goes via the the following observation. By Nemirovski's aforementioned theorem, the boundary of a smooth rationally convex domain is a hypersurface of contact type with respect to the symplectic form $\omega$ granted by the theorem. However, $\om$ can be chosen to be $\omega_{std}$ outside a large ball, and thus, by a result of Gromov, $\omega$ is symplectomorphic to $\omega_{std}$. It follows that the boundary of a smooth rationally convex domain in $\mathbb C^2$ is diffeomorphic to a hypersurface of contact type in $(\mathbb R^4, \omega_{std})$. They then show that there are obstructions to realizing Brieskorn spheres as hypersurfaces of contact type in $(\mathbb R^4, \omega_{std})$. We expect that the results of this paper could be used to obtain further results in this direction.

\noindent {\bf Acknowledgments.} This paper was initiated at the Banff International Research Station (BIRS), Canada, at the conference {\em Interactions between Symplectic and Holomorphic Convexity in $4$ dimensions}. The authors are grateful to the organizers for inviting them. The second author is partially supported by the Infosys Young Investigator Award and the DST FIST Program - 2021 [TPN-700661]. The third author is partially supported by Natural Sciences and Engineering Research Council of Canada.

\section{Convexity on projective and Stein manifolds}\label{S:prelim1}

In analogy with $\Cn$, we say that a compact set $K$ in a complex manifold $X$ is {\em convex with respect to complex hypersurfaces}, or simply, {\em hypersurface convex} if for every $p\in X\setminus K$, there is a complex hypersurface (or effective divisor) that passes through $p$ and avoids $K$.

In the case when $X$ is a projective manifold, we follow ~Guedj~\cite{Gu99} and consider a stronger notion of convexity in this paper: a compact set $K$ in a projective manifold $X$ is said to be {\em rationally convex} if for any point $p$ in $X \setminus K$ there exists a positive
 divisor (i.e., the zero locus of a global holomorphic section of a positive line bundle on $X$) passing through $p$ and avoiding $K$. In general, the {\em rationally convex hull} of a compact set $K$ is the set
	\bes
		r(K)=\{z\in X:\text{every positive divisor passing through $z$ intersects $K$}\}.
	\ees
The advantage of working with convexity with respect to positive divisors is the availability of H\"ormander's $L^2$-methods. Using this notion, Guedj~\cite{Gu99} generalized several results of Duval and Sibony to projective manifolds, including the aforementioned characterization of rationally convex totally real submanifolds, and an approximation theorem for positive closed $(1,1)$-currents.

It is worth distinguishing the two notions of convexity described above. While all effective divisors are positive on $\mathbb{CP}^n$, there are projective manifolds on which rational convexity is a genuinely stronger notion. For example, the compact set $K = \mathbb C\mathbb P^1 \times \{p\}$ in $X=\mathbb C\mathbb P^1 \times \mathbb C\mathbb P^1$ is hypersurface convex, but any positive divisor on $X$ necessarily intersects $K$. In fact, the following is true: if $K$ is hypersurface convex on a projective manifold $X$, but not rationally convex, then the rationally convex hull of $K$ is all of $X$. Indeed, suppose $K$ is convex with respect to hypersurfaces and that there is a section $s_1$ of a positive line bundle $L_1$ whose zero set avoids $K$. Choose a point
$p\in X \setminus K$. There is a line bundle $L_2$ with a section $s_2$ that passes through $p$ and avoids $K$. Then consider the section $s=s_1^Ms_2$ for some positive integer $M$. This is a section of $L_1^M\otimes L_2$ whose zero set passes through $p$ and still avoids $K$, and by choosing $M$ large enough the line bundle $L_1^M\otimes L_2$ can be made positive. This shows that $K$ is rationally convex, which proves our claim. Therefore, in
{\color{black} Theorems~\ref{T:main} and~\ref{T:OkaWeil}}
one may assume that $K$ is hypersurface convex, and the rationally convex hull is not all of $X$. We also note that, in functional terms, rational convexity in $X$ is stronger than convexity with respect to rational (or meromorphic) functions on $X$, as demonstrated by the example of $\mathbb C\mathbb P^1 \times \{p\}$ in $\mathbb C\mathbb P^1 \times \mathbb C\mathbb P^1$. See \cite[Lemma 2.2]{Gu99} for a functional description of rational convexity in projective manifolds.

Now, let $X$ be a Stein manifold. Since every effective divisor is positive on a Stein manifold, the two notions described above coincide, and we simply refer to it as hypersurface convexity in this case. (In Guedj~\cite{Gu99}, this notion still bears the name ``rational convexity".) However, owing to the presence of entire functions, one may define a stronger notion of convexity as follows: we say that $K$ is {\it convex with respect to principal hypersurfaces} if through every point in $X\setminus K$, there exists a principal hypersurface, i.e., the zero locus of an entire function, that avoids $K$. These definitions are inequivalent precisely when $\text{Hom}\left(H_2(X,\mathbb{Z}),\mathbb{Z}\right)\ne 0$, see Col{\c t}oiu~\cite{Col}. If $X$ is a properly embedded submanifold of $\mathbb C^N$, it can be shown (see Boudreaux--Shafikov~\cite{BS}) that a compact $K \subset X$ is convex with respect to principal hypersurfaces in $X$ if and only if $K$ is rationally convex in $\mathbb C^N$.

One may also consider convexity with respect to positive closed currents of bidegree $(1,1)$. For $X=\Cn$, Duval and Sibony show that this notion is equivalent to hypersurface convexity. An analogous statement also holds for a general Stein manifold $X$: a compact set $K\subseteq X$ is hypersurface convex if and only if for every $x\in X\setminus K$, there is a positive closed continuous current $T$ of bidegree $(1,1)$ with $[T]\in H^2(X,\Z)$ such that $x\in\mathrm{supp}\:T$ and $K\cap\mathrm{supp}\: T=\varnothing$. This follows from Proposition~\ref{P:Guedj2.7}(i) below and Theorem~5.7 in Guedj~\cite{Gu99}.

\section{Technical preliminaries}

Let $X$ be a complex manifold. We say that a cohomology class in $H^p(X,\rl)$, $p\geq 0$, is an {\em integral class} if it lies in the image of the morphism $H^p(X,\Z)\rightarrow H^p(X,\rl)$ induced by the containment $\Z\hookrightarrow\rl$. Given a closed $p$-form or $p$-current $T$, this is  abbreviated as $[T]\in H^p(X,\Z)$. A {\em Hodge form} is a K{\"a}hler form whose cohomology class is an integral class. A {\em trivial Hodge form} is a K{\"a}hler form whose cohomology class is trivial.

For the proof of Theorems \ref{T:main} and \ref{T:stein} we will need some technical results that we collect in this section. Although the proof techniques are already present in \cite{DuSi95} and \cite{Gu99}, we provide the proofs here for the sake of completeness.

\begin{lemma}\label{L:guedj} Let $V\subset X$ be an open set, where $X$ is either a projective manifold or a Stein manifold. Let $L$ be a positive holomorphic line bundle on $X$, and $\varphi$ be a positive continuous metric of $L$ on $X$. Let $s$ be a holomorphic section of $L|_{V}$. Suppose $K=\{z\in V:\Vert s(z)\Vert_{\varphi}\geq 1\}$ is compact. Then, for every $a\notin K$, there is an integer $M>0$ and a global holomorphic section $S$ of $L^M$ such that $S(a)=0$ and $K\cap S^{-1}(0)=\varnothing$.
\end{lemma}
In the above statement, $L|_V$ denotes the pull-back bundle $\iota^*(L)$, where $\iota:V\hookrightarrow X$ is the inclusion map, and $\Vert s(z)\Vert_{\varphi}$ denotes $|s(z)|e^{-\varphi(z)}$.

\begin{proof}
When $X$ is projective, the result is proved in Guedj~\cite[Lemma~2.4]{Gu99}. We provide a similar argument for the case of Stein manifolds.

Suppose $X$ is an $n$-dimensional Stein manifold. Fix a K\"ahler form $\omega=dd^c\rho$, where $\rho$ is a strictly plurisubharmonic function on $X$. Fix a point $a\in X\setminus K$.  Using H{\"o}rmander's theorem for the $\bar\partial$-problem for bundle-valued forms (see \cite[Theorem 3.1]{De92b}), we will construct a global holomorphic section $S$ of $L^M$, for some integer $M>0$ to be determined later, such that $S(a)=0$ but $S$ is nonvanishing on $K$.

	Let $\chi\in\cont^\infty_0(X)$ be such that $\mathrm{supp}\:\chi\subseteq V\setminus\{a\}$, $0\leq\chi\leq 1$, and $\chi\equiv 1$ in a neighborhood of $K$. Set $v=\bar\partial\big(\chi s^M\big)=\bar\partial\chi\cdot s^M$. This is can be viewed as a smooth $\bar\partial$-closed $(0,1)$-form with values in $L^M$, or alternatively, as a smooth $\bar\partial$-closed $(n,1)$-form with values in $L^M\otimes K^*_X$, where $K^*_X$ is the dual of the canonical bundle $K_X$.

	Since $X$ is Stein, there exist $h_1,\ldots,h_n\in\mathcal{O}(X)$ such that $h_j(a)=0$ and $\bigcap_{j=1}^{n}\{h_j=0\}=\{a\}$. Let $\rho$ be a strictly plurisubharmonic function on $X$, and $\sigma$ be a smooth positive metric on $K^*_X$. Then $\psi=M\varphi+\sigma+n\log\left(\sum_{j=1}^{n}|h_j|^2\right)+\rho$
is a singular metric on $L^{M}\otimes K^*_X$ that is continuous on $X\setminus\{a\}$, has a logarithmic singularity of order $n$ at $a$, and satisfies $dd^c\psi\geq \om=dd^c\rho$. Since $v$ is compactly supported away from $a$, $v\in L^2_{(n,1)}(X,L^M\otimes K^*_X,e^{-\psi})$. Thus, by \cite[Theorem~3.1]{De92b}, there is an $(n,0)$-form $u$ with values in $L^M\otimes K_X^*$ such that $\bar\partial u=v$ and
\[
	\int_X|u|^2e^{-2\psi}\om^n\leq C\int_X|v|^2e^{-2\psi}\om^n,
\]
for some constant $C>0$. Since the integral on the right converges, it must be  that $u(a)=0$. Viewing $u$ as a section of $L^M$, we have that $S=\chi s^M-u$ is a holomorphic section of $L^M$ such that $S(a)=0$.

	We now show that for sufficiently large $M$, $S$ is nonvanishing on $K$. Since $\bar\partial\chi\equiv 0$ on a neighborhood of $K$, there is an $\alpha\in (0,1)$ such that $|s|e^{-\varphi}\leq\alpha<1$ on $\mathrm{supp }\:\bar\partial\chi$. Let $\beta>0$ such that $\alpha e^{\beta}<1$. For each $y\in K$, let $B(y,r)$ denote the pull-back under a coordinate chart of a Euclidean ball of radius $r$ centered at $y$. Choose $r$ small enough so that $a\notin B(y,2r)$, $B(y,2r)\subset X\setminus\text{supp }(\bar\partial\chi)$ for all $y\in K$, and $|\psi(y)-\psi(z)|<M\beta$ for all $z\in B(y,2r)$. Since $u$ is holomorphic on $B(y,2r)$ and $\psi-M\varphi$ is continuous on $\mathrm{supp}\:\bar\partial\chi$,
\begin{align*}
	|u(y)|^2\lesssim\int_{B(y,r)}|u|^2\om^n
	&\lesssim e^{2(\psi(y)+M\beta)}\int_{B(y,r)}|u|^2e^{-2\psi}\om^n\\
	&\lesssim e^{2M\varphi(y)}e^{2M\beta}
	\int_{\mathrm{supp}\:\bar\partial\chi}|v|^2e^{-2M\varphi}\om^n\\
	&\lesssim e^{2M\varphi(y)}e^{2M\beta}\int_{\mathrm{supp}\:\bar\partial\chi}
		|s^M|^2e^{-2M\varphi}\om^n\\
	&\lesssim  e^{2M\varphi(y)}\big(\alpha e^{\beta}\big)^{2M},
\end{align*}
where the omitted constants are independent of $y$ and $M$. Thus, $\sup_K|u(y)|e^{-M\varphi(y)}\rightarrow 0$ as $M\rightarrow\infty$. On the other hand, $|s|^{M}e^{-M\varphi}\geq 1$ on $K$. Thus, for large $M$, $S=\chi s^M-u$ is nonvanishing on $K$. \end{proof}

\begin{Prop}\label{P:Guedj2.7}
	Let $X$ be either a projective manifold equipped with a K{\"a}hler metric $\om$, or a Stein manifold equipped with a K{\"a}hler metric $\om=dd^c\color{black}\rho$ for some {\color{black} spsh} ${\color{black}{\rho}}\in\cont^\infty(X)$. Let $K\subset X$ be a compact that is rationally convex if $X$ is projective, or  hypersurface convex if $X$ is Stein.
\begin{enumerate}
\item [(i)] For every $z\notin K$, there exists a positive closed current $T$ of bidegree $(1,1)$ on $X$ which admits a continuous potential $\psi$ {\color{black} (i.e., $dd^c \psi = T$ in the sense of distributions for some continuous metric $\psi$ of some line bundle)}, is smooth and strictly positive at $z$, vanishes in a neighborhood of $K$, and $[T]\in H^2(X,\Z)$.  If $K$ is convex with respect to principal hypersurfaces, then $[T]=0$.
\item [(ii)] For every $\eps>0$ and relatively compact neighbourhood $V$ of $K$, there exists a smooth closed $(1,1)$-form $\om_\eps$ which satisfies the following properties:
    \begin{enumerate}
        \item [(a)] $\om_\eps\geq \om$ on $X\setminus V$,
        \item [(b)] $\om_\eps\equiv 0$ in a neighborhood of $K$,
        \item [(c)] $\om_\eps\geq -\eps\om$ in $V$,
        \item [(d)] $[\om_\eps]\in H^2(X,\Z)$. If $K$ is convex with respect to principal hypersurfaces, then $[\om_\eps]=0$.
    \end{enumerate}
\end{enumerate}
\end{Prop}

In the case when $X$ is a projective manifold, the above result is Proposition~2.7 in \cite{Gu99}. 
The proof of (ii) presented in \cite{Gu99} is terse, and merely refers to Richberg's regularization technique. Although this technique appears in various forms in the literature, it is difficult to locate a precise result that grants (ii). For the sake of completeness, we provide a detailed proof of (ii) via a regularization result in the spirit of Lemma~2.15 in \cite{De92} and Theorem~2 in \cite{BlKo07}.

Recall that a function $\phi:X\rightarrow\rl\cup\{-\infty\}$ on a complex manifold $X$ is said to be {\em quasi-plurisubharmonic} (qpsh) if it is locally the sum of a smooth function and a plurisubharmonic function. Given a  continuous $(1,1)$-form $\alpha$ on $V$, $\phi$ is said to be {\em $\alpha$-plurisubharmonic} ($\alpha$-psh) if $\phi$ is qpsh and  $\alpha+dd^c\phi\geq 0$ in the weak sense of currents. The class of $\alpha$-psh functions on $X$ is denoted by $PSH(X,\alpha)$. Given $\varphi_1,...,\varphi_p\in PSH(X,\alpha)$, a regularized maximum of $\varphi_1,...,\varphi_p$ is a function of the form $z\mapsto M_{\eta}(\varphi_1(z),...,\varphi_p(z))$, $z\in X$, where
       \beas
        M_\eta(t_1,...,t_p)
        &=&
\int_{\rl^p}\max\{t_1+\eta_1s_1,...,t_p+\eta_ps_p\}\prod_{j=1}^p\theta(s_j)ds_1...ds_p.
    \eeas
for some $\eta=(\eta_1,...,\eta_p)\in(0,\infty)^p$, and nonnegative $\theta\in\cont^\infty(\rl)$ with support in $[-1,1]$, $\int_\rl\theta(h)=1$ and $\int_\rl h\theta(h)dh=0$. It is easy to check that
\begin{enumerate}
\item [(a)] when $p=1$, $M_\eta(t)=t$,
\item [(b)]  $\max\{t_1,...,t_p\}\leq M_\eta(t_1,...,t_p)\leq \max\{t_1+\eta_1,...,t_p+\eta_p\}$.
\end{enumerate}
The following result provides a sufficient condition for patching up smooth $\alpha$-psh functions.

\begin{lemma}[{\cite[Corollary~5.19]{De12}}]\label{L:patch} Let $X$ be a complex manifold, $\alpha$ be a continuous $(1,1)$-form on $X$, and $\{\Om_j\}_{j\in\N}$ be a locally finite open covering of $X$ consisting of compactly contained open sets. Let $\varphi_j\in\cont^\infty(\overline{\Om_j})\cap PSH(\Om_j,\alpha)$ and $\eta_j>0$, $j\in\N$, be such that
	\begin{enumerate}
\item [(i)] $\varphi_k(z)<\max_{j:\Om_j\ni z}\{\varphi_j(z)\}$,
\item [(ii)] $\varphi_k(z)+\eta_k\leq \max_{j:\Om_j\ni z}\{\varphi_j(z)-\eta_j\}$,
\end{enumerate}
for all $k\in\N$ and $z\in b\Om_k$. Then the function
	\be\label{E:regmax}
		\wt\varphi(z)=M_{(\eta_j)}(\varphi_j(z)),\qquad j\ \text{such that }z\in\Om_j,
	\ee
is smooth and $\alpha$-psh on $X$.
\end{lemma}

The following result is a minor variation of \cite[Lemma~2.15]{De92} and \cite[Theorem~2]{BlKo07}.

\begin{lemma}\label{L:reg} Let $X$ be a complex manifold and $\alpha$ be a continuous $(1,1)$-form on $X$. Let $\varphi$ be a continuous $\alpha$-psh function on $X$ such that $\varphi$ is smooth on some open set $U\subset X$. Then, for any Hermitian metric $\om$ on $X$, $\de>0$, and open subset $U'\Subset U$, there is a smooth $(\alpha+2\de\om)$-psh function $\wt\varphi$ such that $\varphi=\wt \varphi$ on $U'$ and $\varphi\leq \wt\varphi\leq \varphi+2\de$ on $X$.
\end{lemma}

\begin{proof} Let $\Om_0=U$ and $u_0=\varphi|_{U}$. Let $\{\Om_j\}_{j\in\N_+}$ be a locally finite open cover of $X\setminus U$ such that, for each $j\in\N_+$, there is a coordinate patch $(U_j,\Psi_j)$ so that $\Om_j\Subset U_j$ and $\Psi_j(\Om_j)$ is the standard unit ball $\mathbb B^n$ in $\Cn$. Let $0<r<s<1$, $\Om_j''=\Psi_j^{-1}(r\mathbb B^n)$ and $\Om_j'=\Psi_j^{-1}(s\mathbb B^n)$, $j=1,...,k$. By the argument provided in \cite[\S2, pg. 9]{De92}, $\{\Om_j\}_{j\in\N_+}$ can be chosen so that
\begin{itemize}
\smallskip
\item [(i)] $X\setminus U\subset\bigcup_{j\in\N_+}\Om_j''
		\subset \bigcup_{j\in\N_+}\Om_j\subset	X\setminus U'$,
\smallskip
\item [(ii)] the set $\{j\in\N_+:bU\cap \Om_j''\neq\varnothing\}$ is finite, and
\smallskip
\item [(iii)] there exist $f_j\in\cont^\infty(\overline{\Om_j})$ such that $
		\alpha\leq dd^cf_j\leq \alpha+\de\om$ in a neighborhood of $\overline{\Om_j}$.
\end{itemize}
We now produce $\varphi_j\in\cont^\infty(\overline{\Om_j})\cap PSH(\Om_j,\alpha+\de\om)$ satisfying the conditions of Lemma~\ref{L:patch}. Let $\{\rho_\eps\}_{\eps>0}$ be a family of smoothing kernels on $\Cn$. It is known, see, e.g.,~\cite[Theorem 2.9.2]{K91} that
if $\Omega\subset\mathbb{C}^n$ is open, $u\in PSH(\Omega)$, $\varepsilon>0$, and $\Omega_\varepsilon=\{z\in\Omega\,:\,\text{dist}(z,b\Omega)>\varepsilon\}\neq\varnothing$, then $\{u*\rho_\varepsilon\}_{\varepsilon>0}\subset\mathcal{C}^{\infty}\cap PSH(\Omega_{\varepsilon})$ monotonically decreases to $u$ locally uniformly as $\varepsilon\to 0$.
It follows then from (iii) that
$$
u_j =\varphi+f_j\in\cont(\overline{\Om_j})\cap PSH(\overline{\Om_j}).
$$
Let $u_{j,\eps}$ be the regularization of $u_j$ given by
	\bes
		u_{j,\eps}=\left[ \left(u_j\circ\Psi^{-1}_j\right)\ast\rho_\eps
						\right]\circ\Psi_j,
	\ees
for sufficiently small $\eps$. Then, $\{u_{j,\eps}-f_j\}_{\eps>0}$ is a decreasing family of smooth functions in $PSH(\overline{\Om_j},\alpha+\de\om)$ that converges uniformly to $u_j-f_j=\varphi$ on  $\overline{\Om_j}$ as $\eps\rightarrow 0$.
For $j\in\N_+$, let
    \bes
        \varphi_j(z)=u_{j,\eps_j}(z)-f_j(z)+\de_j(s^2-|\Psi_j(z)|^2),
\qquad z\in\overline{\Om_j},
    \ees
for $\eps_j$ and $\de_j$ small enough so that $\varphi_j\leq \varphi+\de$ and $\alpha+dd^c \varphi_j\geq -2\de\om$. Let
   \bes
        \eta_j=\begin{cases}
            \de_j\min\left\{(s^2-r^2)/2,(1-s^2)/4\right\},&
				\text{if }j\in\N_+,\\
            \min\{\eta_j:bU\cap\Om_j''\neq\varnothing,\ j\geq 1\},& \text{if }j=0.
    \end{cases}
    \ees
Further shrink $\de_j$ and $\eps_j$ so that for all $j\in\N_+$, $\eta_j<\de$, and $u_{j,\eps_j}<u_j+2\eta_j$ on $\overline{\Om_j}$, i.e.,
\bes
\varphi=u_j-f_j\leq u_{j,\eps_j}-f_j<u_j-f_j+2\eta_j=\varphi+2\eta_j
\qquad \text{on }\overline{\Om_j}.
\ees
Since $\de_j(s^2-|\Psi_j(z)|^2)\leq -4\eta_j$ on $b\Om_j$ and $\de_j(s^2-|\Psi_j(z)|^2)>2\eta_j$ on $\overline{\Om_j''}$, we have that
    \bea
        \varphi_j&<&\varphi-2\eta_j\ \text{on } \ b\Om_j,\label{E:bdy}\\
        \varphi_j&>& \varphi+2\eta_j\ \text{on } \ \overline{\Om_j''}.
			\label{E:int}
        \eea
Thus, if $z\notin \overline{\Om_0}=\overline U$, then Condition (ii) of 	Lemma~\ref{L:patch} is satisfied.

If $z\in b\Om_0=bU$, then by \eqref{E:int}, $\varphi_0(z)+\eta_0=\varphi(z)+\eta_0<\varphi_j(z)-\eta_j$ for any $j$ such that $z\in\overline{\Om_j''}$.
 On the other hand, if $z\in\Om_0$ and $z\in b\Om_j$ for some $j>0$, then by \eqref{E:bdy}, $\varphi_j(z)+\eta_j<\varphi(z)-\eta_j\leq\varphi_0(z)-\eta_0$. Thus, the collection $\{\Om_j,\varphi_j,\eta_j\}_{j\in\N_+}$ satisfies the hypothesis of Lemma~\ref{L:patch} for the continuous $(1,1)$-form $\alpha+2\de\om$, and
 $\wt\varphi(z)$ as given in \eqref{E:regmax}
	is smooth and $(\alpha+2\de\om)$-psh on $X$. By property (a) of regularized maxima, and the fact that $U'\cap\Om_j=\varnothing$ for all $j\neq 0$, we have that $\wt\varphi=\varphi$ on $U'$. By property (b) of regularized maxima, and the fact that $\varphi(z)\leq \max_{j:\Omega_j\ni z} \varphi_j(z)$ and $\max_{j:\Omega_j\ni z}(\varphi_j(z)+\eta_j)\leq \varphi(z)+2\de$, we have that $\varphi\leq\wt\varphi\leq \varphi+2\de$.
\end{proof}

\noindent{\em Proof of Proposition~\ref{P:Guedj2.7}}. (i) Let $z\notin K$. Then there is a positive holomorphic line bundle $L$, an $s\in\Gamma(X,L)$, and a smooth positive metric $G$ of $L$ on $X$ such that $s(z)=0$, $K\cap s^{-1}(0)=\varnothing$ and $|s|e^{-G}>1$ on $K$. Set $\psi=\max\{\log|s|,G\}$. Then $T=dd^c\psi$ is the desired current. When $K$ is convex with respect to principal hypersurfaces, $L$ can be chosen to be the trivial bundle, and hence $[T]=0$.

(ii) Fix an $\eps>0$ and some relatively compact neighborhood $V$ of $K$. Let $W$ be an open neighborhood of $K$ such that $W\Subset V$. When $X$ is Stein, let $\Om_1\Subset \Om_2$ be relatively compact open subsets of $X$ such that $V\Subset \Om_1$. Let $B_j=\overline{\Om_j}$, $j=1,2$. When $X$ is a projective manifold, let $B_1=B_2=X$. Since $B_2\setminus W$ is compact, Part (i) gives finitely many nonnegative closed currents $T_1,...,T_k$ of bidegree $(1,1)$ that admit continuous potentials, vanish on a neighborhood of $K$, have integral cohomology class, and have sum which is strictly positive on $B_2\setminus W$. Let $T=\sum_{j=1}^k T_{j}$. Note that $T$ is a closed $(1,1)$ current with the following properties:
        \begin{enumerate}
            \item $T\geq 0$ on $X$,
			\item $T$ is strictly positive on {$B_2\setminus W$}, i.e., $MT\geq\om$ on {$B_2\setminus W$}, for some positive integer $M$,
            \item $T$ admits a local continuous potential everywhere on $X$,
            \item $T\equiv 0$ on some neighborhood $U\Subset V$ of $K$.
        \end{enumerate}
Let $\chi\in\cont^\infty_0(X)$ be such that $0\leq\chi\leq 1$, $\chi\equiv 1$ on {$B_1\setminus V$} and $\chi\equiv 0$ on {$W\cup(X\setminus B_2)$}. Then, $\gamma=\chi\om$ is a smooth $(1,1)$-form on $X$, and
\be\label{E:gamma}
    MT\geq \gamma\ \text{on }X.
\ee

Since the Bott--Chern cohomology of $X$ can be computed either via currents or via smooth forms~\cite[Remarks after Lemma VI.12.2]{De12}, there is a smooth $(1,1)$-form $\beta$ on $X$ such that $[MT]_{BC}=[\beta]_{BC}$, i.e.,
        \bes
            MT=\beta+dd^c\varphi\quad \text{on }X,
        \ees
for some $(0,0)$-current $\varphi$. 
From $(1)$-$(4)$ and \eqref{E:gamma}, it follows that
            \begin{enumerate}
                \item [(i)] $\varphi$ is a continuous function on $X$,
						\item [(ii)] $\varphi$ is $(\beta-\gamma)$-psh, i.e., $\beta+dd^c\varphi\geq\gamma$ on $X$,
                \item [(iii)] $\varphi$ is smooth on $U$.
            \end{enumerate}
Let $\de>0$ such that $\frac{2\de}{1-2\de}<\eps<1$. Then, by Lemma~\ref{L:reg}, there exists a smooth function $\wt\varphi\in PSH(X,\beta-\gamma+2\de\om)$ such that $\wt\varphi=\varphi$ on some neighborhood $U'\Subset U$ of $K$, and
	\bes
		\beta+dd^c\wt\varphi\geq\gamma-2\de\om=(\chi-2\de)\om\quad \text{on } X.
	\ees
Let $N\in\mathbb N$ such that $N>\frac{1}{1-2\de}$, and
    \bes
    \wt\om_\eps=N(\beta+dd^c\wt\varphi)\quad \text{on } X.
    \ees
Suppose $X$ is projective. Let $\om_\eps=\wt\om_\eps$. Then, (a) holds because $N(\beta+dd^c\wt\varphi)\geq \frac{1}{1-2\de}(\om-2\de\om)=\om$ on $X\setminus V$. Since $\om_\eps=NMT\equiv 0$ on $U'$, (b) holds. Claim (c) holds because, in $V$, $N(\beta+dd^c\wt\varphi)\geq -\frac{2\de}{1-2\de}\om>-\eps\om$. Lastly, (d) holds since $\om_\eps=N\beta=NM[T]\in H^2(X,\Z)$.

Suppose $X$ is Stein. Then $\wt\om_\eps\geq\frac{1}{1-2\de}(\om-2\de\om)=\om$ on $B_1\setminus V$. We obtain $\om_\eps$ by modifying $\wt\om_\eps$ outside $B_1$. Let $\lambda\in\cont_0^\infty(X)$ be such that $\text{supp}(\lambda)\subseteq B_1$, $0\leq\lambda\leq 1$, and $\lambda\equiv 1$ on some neighborhood of $\overline V$. Then, for sufficiently small $\eps'>0$,
	\bes
		\om_\eps=\wt\om_\eps+dd^c(\eps'(1-\lambda)\psi)
	\ees
satisfies $(a)$-$(d)$ since $\om_\eps=\wt\om_\eps$ on some neighborhood of $\overline V$. If, furthermore, $K$ is convex with respect to principal hypersurfaces, then $[\om_\eps]=0$ since $[T]=0$ in (i).
\qed



\section{Proofs of Theorems~\ref{T:main} and \ref{T:stein}}

Let $X$ be either a projective or a Stein manifold.
Suppose there is an open neighborhood $V\Subset U$ of $K$ and a Hodge form $\om$ on $X$ such that $\om=dd^c\varphi$ on $V$. By Sard's theorem, there exist $0<\eps'<\eps$ so that $D_{\eps}=\{z\in U\,:\,\varphi(z)<\eps\}\Subset V$, and both $D_{\eps'}$
and $D_{\eps}$ are smoothly bounded Stein domains. Since $D_\eps$ is smooth, $H_p(D_\eps,\Z)$ and $H^p(D_\eps,\Z)$ are finitely generated for all $p$, and the kernel of the morphism $H^2(D_\eps,\Z)\rightarrow H^2(D_\eps,\rl)\cong H^2(D_\eps,\Z)\otimes\rl$ given by $a\mapsto a\otimes 1$ is the torsion subgroup of $H^2(D_\eps,\Z)$. Moreover, since $D_\eps$ is Stein, $Pic(D_\eps)\cong H^2(D_\eps,\Z)$ via $L\mapsto c_1(L)$.

Let $(L,\sigma)$ be a positive Hermitian line bundle on $X$ such that $dd^c\sigma=\omega$, and thus, $c_1(L)=[\omega]$; see \cite[Theorem V.13.9]{De12}. Since 	$c_1(L|_{D_{\eps}})=[\iota^*_{D_{\eps}}\omega]=[dd^c\varphi]
		=0$ in $H^2(D_\eps,\rl)$, there is some $M\in\N$ so that $Mc_1(L|_{D_{\eps}})=c_1(L^M|_{D_{\eps}})=0$ in $H^2(D_\eps,\Z)$. Thus, $L^M|_{D_\eps}$ is trivial, and the restriction of the metric $\psi=M\sigma$ to $D_\eps$ may be identified with a function, also denoted by $\psi$, via a trivialization of $L^M$ over $D_\eps$. Since the curvature form of a metric is independent of the choice of trivializations, the function $h=\psi-M\varphi\in\text{PH}(D_{\eps})$, the space of pluriharmonic functions on $D_\eps$.

Let $\{\gamma_1,\ldots,\gamma_p\}$ be a $\Z$-basis of $H_1(D_\eps,\mathbb{Z})$, and $\{\nu_1,...,\nu_p\}\subset H^1(D_\eps,\rl)
$ be the dual basis, i.e.,  $\nu_k(\gamma_j)=\de_{jk}$, $1\leq j,k\leq p$,
where $\de_{jk}$ denotes the Kronecker symbol. By \cite[Lemma~1.4]{Gu99}, there is a surjective map $\Phi:PH(D_\eps)\rightarrow H^1(D_\eps,\rl)$ such that for $g\in PH(D_\eps)$,
\be\label{E:Phi}
\text{$g=\log|s|$ for some $s\in\hol^*(D_\eps)$ if and only if  $\Phi(g)(\gamma_j)\in \Z$ for all $j=1,...,p$}.
\ee
By the surjectivity of $\Phi$, there exist $h_1,\ldots h_p\in\text{PH}(D_\eps)$ such that $\Phi(h_j)=\nu_j$, $j=1,...,p$. Thus, there is a dense set $\Lambda$ of vectors $(\lambda_1,...,\lambda_p)\in\rl^p$ such that
\be\label{E:lambda}
	\Phi(h+\lambda_1h_1+\cdots+\lambda_ph_p)(\gamma_j)\in \mathbb Q,\quad \forall j=1,...,p.
\ee
Let $\chi\in\mathcal{C}^{\infty}_0(D_\eps)$ be such that $\chi\equiv 1$ on $D_{\eps'}$. Let $(\lambda_1,...,\lambda_p)\in\Lambda$ be so small in norm so that the function
\bes
	\wt\psi=\psi+\chi\left(\lambda_1h_1+\ldots+\lambda_ph_p\right)
\ees
is strictly plurisubharmonic on $D_\eps$. Since $\wt\psi$ coincides with $\psi$ outside a compact subset of $D_\eps$, $\wt\psi$ extends to a positive Hermitian metric on $L^M$, also denoted by $\wt\psi$. The proof will be completed by applying Lemma~\ref{L:guedj} to $\left(L^N,N\wt\psi\right)$ for sufficiently large $N\in\N$.

For this, observe that by \eqref{E:lambda} and \eqref{E:Phi}, there exists an $N\in\N$ and $s\in\hol^*(D_\eps)$ such that $|s|=e^{N(h+\lambda_1h_1+\cdots+\lambda_ph_p)}$ on $D_\eps$. Now, $s$ can be viewed as a holomorphic section of $L^{MN}|_{D_\eps}$ via the $N^{\text{th}}$ power of the same trivialization of $L^M$ as considered earlier. Thus, we have that
	\beas
	K&=&\{z\in D_{\eps'}:\varphi(z)\leq 0\}\\
	&=&\{z\in D_{\eps'}:e^{M\varphi(z)}\leq 1\}\nonumber\\
	&=&\left\{z\in D_{\eps'}\,:\,e^{\psi(z)}\leq e^{h(z)}\right\}\\
	&=&\left\{z\in D_{\eps'}:
	e^{N(\psi(z)+\lambda_1h_1(z)+\cdots+\lambda_ph_p(z))}
		\leq e^{N(h(z)+\lambda_1h_1(z)+\cdots+\lambda_ph_p(z))}\right\}\\
	&=&\left\{z\in D_{\eps'}\,:\,e^{N\wt\psi(z)}\leq |s(z)|\right\}
	=\left\{z\in D_{\eps'}:\Vert s(z)\Vert_{-N\wt\psi}\geq 1\right\}.
\eeas
Thus, by Lemma~\ref{L:guedj}, $K$ is rationally (resp. hypersurface) convex in $X$. In the case when $[\om]=0$, $L$ is the trivial bundle. Thus, by Lemma~\ref{L:guedj}, for every $a\notin K$, there is a holomorphic function $f$ on $X$ whose zero locus separates $a$ from $K$. Thus, $K$ is convex with respect to principal hypersurfaces.

Conversely, suppose $K$ is a rationally or hypersurface convex compact. Fix a K\"ahler metric $\omega$ on $X$, and a neighborhood $V$ of $K$ such that $V\Subset U$.
If $X$ is Stein, in addition choose $\omega = dd^c\rho$ for some strictly plurisubharmonic function $\rho$ on $X$.
Let $\chi\in \cont_0^\infty(U)$ be such that $\chi\equiv 1$ on $V$. Then $\chi\varphi$ extends to a smooth function on $X$, which we also denote by $\chi\varphi$. Since $\varphi$ is strictly plurisubharmonic on $U$ and $\text{supp }\chi$ is compact, there is an $M\in\N$ large enough so that  $dd^c(\chi\varphi)=dd^c\varphi\geq\frac{1}{M}\omega$ on $V$ and $
	dd^c\left(\chi\varphi\right)\geq -M\omega$ on $X\setminus V.$

Let $\eps=-\frac{1}{4M^2}$. By Proposition~\ref{P:Guedj2.7}, there is a smooth $(1,1)$-form $\omega_{\eps}$ on $X$ such that
\begin{itemize}
	\item[(a)] $\omega_{\eps}\geq\omega$ on $X\setminus V$,
	\item[(b)] $\omega_{\eps}\equiv 0$ in a neighborhood of $K$,
	\item[(c)] $\omega_{\eps}\geq -\frac{1}{4M^2}\cdot\omega$ in $V$,
	\item[(d)] $[\omega_{\eps}]\in H^2(X,\mathbb{Z})$.
\end{itemize}
Define $\wt\omega(z)=dd^c\left(\chi\varphi\right)+2M\omega_{\eps}$.
	On a neighborhood of $K$, $\wt\omega=dd^c\varphi$. When $z\in X\setminus V$, $\wt\omega=dd^c\left(\chi\varphi\right)+2M\omega_{\eps}
\geq M\omega>0.$ When $z\in V$, $	\wt\omega=dd^c\varphi+2M\omega_{\eps}
\geq\frac{1}{2M}\omega>0.$ Thus, $\wt\om$ is a K{\"a}hler form. Lastly, note that $[\wt\omega]=[M\omega_{\eps}]\in H^2(X,\mathbb{Z})$, and is $0$ if $K$ is convex with respect to principal hypersurfaces. Thus, $\wt\omega$ is the desired Hodge form on $X$. This completes the proof of Theorems~\ref{T:main} and~\ref{T:stein}.

\section{Proof of Theorem~\ref{T:OkaWeil}}

\begin{lemma}\label{l.mmm}
	Let $s$ be a holomorphic section of a positive line bundle on a complex projective manifold $X$. Then every function holomorphic on $X\setminus s^{-1}(0)$ is the uniform limit on compacts of meromorphic functions with poles on $s^{-1}(0)$.
\end{lemma}
\begin{proof}
	Write $H=s^{-1}(0)$ and let $f\in\mathcal{O}(X\setminus H)$. Since $L$ is positive, the Kodaira embedding theorem gives a $k\in\mathbb{N}$ and a basis $(s_0=s^k,s_1,\ldots,s_N)$ of $\Gamma(X,L^k)$ so that the map $\Phi:X\to\mathbb{CP}^{N}$ given by
\[
	x\mapsto [s_0(x):\ldots:s_N(x)]
\]
	defines a holomorphic embedding of $X$ onto a subvariety $V$ of $\mathbb{CP}^N$ with $L=\Phi^*(\mathcal{O}(1)|_V)$. It follows that $X\setminus H$ maps into $\mathbb{C}^N\cong\mathbb{CP}^N\setminus\{z_0=0\}$. Therefore $f\circ\Phi^{-1}$ is a holomorphic function on a subvariety of $\mathbb{C}^N$ and hence by the Oka--Cartan extension theorem, there exists an $F\in\mathcal{O}(\mathbb{C}^N)$ which restricts to $f\circ\Phi^{-1}$ on $V$. Expanding $F$ into a power series and precomposing its Taylor polynomials by $\Phi$ gives the desired sequence of meromorphic functions.
\end{proof}

\begin{proof}[Proof of Theorem~\ref{T:OkaWeil}]
	Suppose that $K$ is rationally convex and let $U$ denote a neighborhood of $K$ on which $f$ is defined. By the rational convexity of $K$, for every $p\not\in K$, there exists a section $s$ of a positive line bundle $L$ with zero set $h$ such that $p\in h$ and $K\cap h=\varnothing$. Since $X\setminus h$ is a Stein manifold which contains $K$ as a compact subset, it follows that $\widehat{K}_{\mathcal{O}(X\setminus h)}$ (the holomorphically convex hull of $K$ in the Stein manifold $X\setminus h$) is a compact set of $X\setminus h$, and hence there is a neighborhood $V$ of $p$ that is disjoint from $\widehat{K}_{\mathcal{O}(X\setminus h)}$. By compactness we can cover $X\setminus U$ by finitely many neighborhoods $V_1,\ldots,V_k$ with associated hypersurfaces $h_1,\ldots h_k$ coming from global sections $s_1,\ldots, s_k$ of positive holomorphic line bundles $L_1,\ldots, L_k$. Therefore,
\[
	\widehat{K}_{\mathcal{O}(X\setminus h_1)}\cap\ldots\cap\widehat{K}_{\mathcal{O}(X\setminus h_k)}\subset U.
\]
	Set $s^*=s_1\cdots s_k\in\Gamma(X,L_1\otimes\cdots\otimes L_k)$. The zero set of $s^*$ is $h^*=h_1\cup\ldots\cup h_k$, and furthermore
\[
	\widehat{K}_{\mathcal{O}(X\setminus h^*)}\subseteq\widehat{K}_{\mathcal{O}(X\setminus h_1)}\cap\ldots\cap\widehat{K}_{\mathcal{O}(X\setminus h_k)}.
\]
	The given function $f$ is holomorphic on a neighborhood of $\widehat{K}_{\mathcal{O}(X\setminus h^*)}$, so the classical Oka--Weil theorem for Stein manifolds shows that it is the uniform limit on $K$ of a sequence of functions in $\mathcal{O}(X\setminus h^*)$. Via Lemma~\ref{l.mmm} above, elements of $\mathcal{O}(X\setminus h^*)$ can be approximated uniformly on compacts by meromorphic functions on $X$ with poles in $h^*$. This completes the proof.
\end{proof}


\section{Proof of Theorem~\ref{T:char}}

In this section, we use the fact that, for compact subsets of Stein manifolds, convexity with respect to (principal) hypersurfaces is equivalent to (strong) meromoprhic convexity; see \cite[Definition 1.1 \& Proposition 1.2]{BS}. For the sake of convenience, we refer to convexity with respect to (principal) hypersurfaces as (strong) meromorphic convexity throughout this section. By \cite[Lemma 2.2]{Gu99}, a similar functional characterization holds for rational convexity in a projective manifold $X$, where the appropriate class of meromorphic functions is
	\bes
	\mathcal M^+(X)=\left\{\frac{f}{g}: f,g\in\Gamma(X,L)\ \text{for some positive } L\in Pic(X)\right\}.
	\ees
In all three cases, we also need the following generalization of Corollary~1.5.4. in \cite{St07}. The proof is identical to that of Corollary~1.5.4 in \cite{St07}, and relies on the Oka--Weil theorem for (strongly) meromorphically convex sets in Stein manifolds and rationally convex sets in projective manifolds, see \cite[Theorem 2.1]{BS}, \cite[Theorem 3.4]{Ro61}, \cite[Theorem 2]{Hi71}, and Theorem~\ref{T:OkaWeil}.

\begin{lemma}\label{L:conncomp} Let $X$ be a Stein manifold, and $K\subset X$ be a compact subset that is convex with respect to (principal) hypersurfaces. Suppose $K$ is the disjoint union of compact sets $L$ and $M$. Then $L$ and $M$ are convex with respect to (principal) hypersurfaces.
The same result holds for a rationally convex compact subset of a projective manifold.
\end{lemma}

\begin{proof}[Proof of Theorem~\ref{T:char}] Let $X$ be a Stein manifold, and $K\subset X$ be a (strongly) meromorphically convex compact set. Let $U\subset X$ be a relatively compact open set that contains $K$. By the (strong) meromorphic convexity of $K$ and the compactness of $bU$, there exist finitely many (strongly) meromorphic functions $R_1,...,R_m$ on $X$ and an $\eps>0$ so that, for each $p\in bU$, there exists a $j\in\{1,...,m\}$ such that $R_j$ is well-defined on $K\cup\{p\}$ and satisfies
	\be\label{E:rat}
		|R_j(p)|>\Vert R_j\Vert_K+2\eps,
	\ee
where $\Vert R\Vert_K=\sup_{z\in K}|R(z)|$. Given $\eps'\in(0,\eps]$, let
	\bes
		\Om_{\eps'}=\left\{z\in U: R_j\ \text{is well-defined on $z$ and }
			|R_j(z)|<\Vert R_j\Vert_K+\eps',\ \forall j=1,...,m\right\}.
	\ees
Then, owing to \eqref{E:rat}, $\overline \Om_{\eps'}$ is a union of connected components of a (strongly) meromorphically convex compact set in $X$. By Lemma~\ref{L:conncomp}, $\overline\Om_{\eps'}$ is itself a (strongly) meromorphically convex compact in $X$. It is clear that $K\subset\Om_{\eps'}$ for all $\eps'\in(0,\eps]$. Moreover, given any compact set $L\subset\Om_\eps$, there exists an $\eps'\in(0,\eps)$ such that $L\subset\Om_{\eps'}\Subset\Om_\eps$.

Since $\Om_{\eps}$ is an analytic polyhedron, it is Stein. Thus, it admits a strictly plurisubharmonic exhaustion function, say $\rho$. Choose $c\in\rl$ so that
	 \bes
		D= \left\{z\in \Om_{\eps}:\rho(z)<c\right\}
	\ees
is strongly pseudoconvex, and $K\subset D$. Since $\overline D$ is compact, there exist $0<\eps''<\eps'<\eps$ such that $\Om''=\Om_{\eps''}$ and $\Om'=\Om_{\eps'}$ satisfy
	\bes
		K\subset \overline D\subset \Om''\Subset \Om'\Subset\Om_{\eps}.
\ees
Let $\chi\in\cont^\infty_0(X)$ be such that $\text{supp}(\chi)\subset \Om_{\eps}$, $0\leq \chi\leq 1$, and $\chi\equiv 1$ on $\overline{\Om'}$. Let $\om$ be a K{\"a}hler form on $X$ {\color{black} (in the case of strong meromorphic convexity, assume that $\omega = dd^c h$, where $h$ is a strongly psh function on $X$})  such that {\color{black} after scaling $\omega$ if necessary, we have}
	\bes
		\om+dd^c(\chi\rho)>0\text{ on }X.
	\ees
Let $\de>0$ so that $-\de\om>-dd^c\rho$ on $\Om'$. Then, since $\overline{\Om''}$ is (strongly) meromorphically convex, by Proposition~\ref{P:Guedj2.7}, there exists a smooth closed $(1,1)$-form $\wt\om$ such that $\wt\om\equiv 0$ on a neighborhood of $\overline{\Om''}$, $\wt\om\geq -\de\om$ on $\Om'$, $\wt\om\geq\om$ on $X\setminus\Om'$, and $\wt\om$ has (trivial) integer cohomology class. Then,
	\beas
		dd^c(\chi\rho)+\wt\om\geq
		\begin{cases}
			dd^c\rho-\de\om>0,& \text{ on }\Om'\\
		dd^c(\chi\rho)+\om>0,& \text{ on }X\setminus \Om'.
	\end{cases}
	\eeas
Thus, $dd^c(\chi\rho)+\wt\om$ is a (trivial) Hodge form on $X$ that coincides with $dd^c\rho$ on a neighborhood of $\overline D$. Since $U$ was an arbitrary open neighborhood of $K$, we obtain a neighborhood basis with the desired property.

The other direction of the claim follows from the fact the the intersection of an arbitrary family of (strongly) meromorphically convex sets is (strongly) meromorphically convex.

The case when $X$ is a projective manifold can be argued similarly, with the exception that $R_1,...,R_m$ belong to $\mathcal M^+(X)$.
\end{proof}



\end{document}